\documentclass[letterpaper,12pt]{amsart}

\usepackage{amsmath,amsthm,amssymb}
\usepackage{yfonts}
\usepackage{hyperref}
\hypersetup{colorlinks,linkcolor=blue,urlcolor=blue,citecolor=blue}

\usepackage[margin=1in]{geometry}

\newtheorem{theorem}{Theorem}
\newtheorem{proposition}{Proposition}
\newtheorem{lemma}{Lemma}
\newtheorem{corollary}{Corollary}
\newtheorem{remark}{Remark}

\theoremstyle{definition}
\newtheorem{definition}{Definition}
\newtheorem{example}{Example}

\numberwithin{equation}{section}

\let\oldproofname=\proofname
\renewcommand{\proofname}{\bf{\oldproofname}}

\begin{document}

\title{On Kronecker's solvability theorem}

\author{Yan Pan}
\address{Department of Mathematics, Henan Institute of Science and Technology, Henan 453003, People's Republic of China}
\email{ypan@stu.hist.edu.cn}
\thanks{}

\author{Yuzhen Chen}
\address{Department of Mathematics, Henan Institute of Science and Technology, Henan 453003, People's Republic of China}
\email{chenyuzhenjg@nuaa.edu.cn}
\thanks{}

\subjclass[2010]{Primary 12E05, 12F15}
\keywords{Kronecker, D{\"o}rrie, solvable polynomials}

\begin{abstract}
Kronecker's 1856 paper contains a solvability theorem that is useful to construct unsolvable polynomial equations. We show how Kronecker's solvability theorem can be derived naturally via a polynomial complete decomposition method. This method is similar to D{\"o}rrie, but we fill a gap that appears in his proof.
\end{abstract}

\maketitle

\section{Introduction}
For a long time, people attempted to find a formula for the roots of a general quintic polynomial by using a finite combination of coefficients, radicals, and arithmetic operations, but failed. It was not until 1824 Abel gave a proof that such a formula does not exist. Subsequently, in 1830, Galois provided a proof without knowing Abel's work and gave a criterion of the solvability of the general polynomial equation. However, neither Galois nor Abel give any specific example. Kronecker in his 1856 paper~\cite{kronecker1856uber} presented the following Proposition~\ref{pro:kro}, which is easy to use to construct unsolvable examples (see~\cite{dorrie1965100}).

\begin{proposition}[Kronecker]\label{pro:kro}
If an irreducible polynomial with integer coefficients is solvable and its degree is an odd prime, then either all of its roots or only one of them is real.
\end{proposition}

Kronecker proved Proposition~\ref{pro:kro} by the following Proposition~\ref{pro:gal}. For a similar proof, see~\cite{rosen1995niels}. But~\cite{rosen1995niels} mistakenly asserted that ``Kronecker was clearly unaware of Galois' work.''

\begin{proposition}[Galois]\label{pro:gal}
If a solvable irreducible polynomial with rational coefficients is solvable and its degree is an odd prime, then the root of such a polynomial can be represented as a rational function of any two other roots.
\end{proposition}

Kronecker (in~\cite{kronecker1856uber}) explained why no one seems to have come up with a similar proposition before him---Galois' proof is incomplete so that the coefficients in Galois' rational function might contain some irrational quantities like roots of unity. Kronecker also said that he would recently publish his new and simpler method to definite that, for any solvable irreducible equation with real coefficients of odd prime degree,\footnote{Kronecker also stated another proposition where he strengthens ``integer coefficients'' in Proposition~\ref{pro:kro} to ``real coefficients.'' For this reason, Kronecker seems to consider equations with real coefficients at best.} Galois' rational function's coefficients are the rational function of the equation's coefficients. Unfortunately, we do not find it.
 
D{\"o}rrie's book~\cite[\S 25]{dorrie1965100} contains a widely known elemental proof\footnote{An earlier same proof, see Weber's \textit{Enzyklop{\"a}die der Elementarmathematik; Band 2: Algebra}. Some literature seem mistaken to suggest this proof is what Kronecker 1856 gave in~\cite{kronecker1856uber}.} of Proposition~\ref{pro:kro}. That proof may be inspired by Kronecker's formula (IV) in~\cite[p.~213]{kronecker1856uber} or Abel's proof~\cite{abel1824memoire}. The following Theorem~\ref{th:KroEdw} contains a modern version of Kronecker's formula (IV).

\begin{theorem}[\cite{edwards2014roots}, Theorem~3.1's weak version]\label{th:KroEdw}
Let $g \in \mathbb{Q}[x]$ be a solvable irreducible polynomial of odd prime degree $\mu$. We can find an irreducible cyclic polynomial $f \in \mathbb{Q}[x]$ of degree $v \mid (\mu-1)$. Let $r_1, r_2, \ldots, r_{v}$ be all the roots of $f$. We can choose a proper positive integer $\delta$ whose order mod $\mu$ is $v$, and take a $\mu$th root $w$ of $r_{1}^{\delta^{\nu-1}} r_{2}^{\delta^{\nu-2}} r_{3}^{\delta^{\nu-3}} \cdots r_{\nu}$ such that $w$ is not a $\mu$th power of $\mathbb{Q}(r_1,r_2,\ldots,r_v)$. Then, a root $x_g$ of $g$ satisfies
\[
x_g=c_0+c_1 w+c_2 w^2+\cdots+c_{\mu-1} w^{\mu-1},
\]
where $c_0,c_1,\ldots,c_{\mu-1} \in \mathbb{Q}(r_1,r_2,\ldots,r_v)$.
\end{theorem}

By Theorem~\ref{th:KroEdw} and D{\"o}rrie's method, we have that all the roots of $g$ can be expressed by
\begin{equation}\label{eq:KroDor}
x_j = c_0 + c_1 {w {\mathrm{e}}^{\frac{2\pi \mathrm{i} j}{\mu}}} + c_2 {( w {\mathrm{e}}^{\frac{2\pi \mathrm{i} j}{\mu}})}^2 + \cdots + c_{\mu-1} {( w {\mathrm{e}}^{\frac{2\pi \mathrm{i} j}{\mu}})}^{\mu-1},
\end{equation}
where $1 \leq j \leq \mu$. The proof of~\eqref{eq:KroDor} is given in Appendix~\ref{app:1}. 

Formulas similar to~\eqref{eq:KroDor} play a vital role in some proofs of Abel's theorem (Theorem~\ref{th:Abel}), which is also the core of this paper. Section~\ref{sec:2} introduces some definitions and results on radical extensions. Section~\ref{sec:3} proves that Definition~\ref{def:solvable by radicals} is equivalent to the usual ones. The two theorems in Sections~\ref{sec:4} and~\ref{sec:5} can help us find a formula~\eqref{eq:Dorrie} equivalent to D{\"o}rrie's main formula but weak than~\eqref{eq:KroDor}. Then, in Section~\ref{sec:6}, we get Kronecker's solvability theorem by discussing the role of a pair of complex conjugate roots under transitive transformation. Finally, in Section~\ref{sec:7}, we point out a gap in D{\"o}rrie's proof and then give a supplement. 

In particular, we assume that the readers are very fluent in basic field theory and the fundamental theorem on symmetric polynomials. One can learn them from~\cite{nan2009field,weintraub2009galois}.

\section{Some definitions and results on radical extension}\label{sec:2}
Throughout this article, all fields contained in $\mathbb{C}$, and all polynomials are monic.
\begin{definition}\label{def:irreducible radical tower}
If $p$ is prime or $1$, a field extension $D \subseteq E$ such that $E=D(u)$, where $x^{p}-u^{p} \in D[x]$ is irreducible over $D$, we call $D \subseteq E$ satisfies condition $(u,p)$. A field extension $D_{0} \subseteq D_{k}$ is said to be an \textit{irreducible radical tower} if there is a series of intermediate fields
\[
D_0 \subseteq D_1 \subseteq D_2 \subseteq \cdots \subseteq D_{l-1} \subseteq D_l \subseteq \cdots \subseteq D_k
\]
such that $D_{j-1} \subseteq D_{j}$ satisfies $(u_{j},b_{j})$ for $1 \leq j \leq k$. We call this series satisfies $(u_{j},b_{j})_{1}^{k}$.
\end{definition}

The following Definition~\ref{def:solvable by radicals} is equivalent to~\cite[Definition~4.3.1]{weintraub2009galois}. We prove it in Section~\ref{sec:3}.

\begin{definition}\label{def:solvable by radicals}
We call $f \in D_0 [x]$ \textit{solvable by radicals} over field $D_0$ if all the roots of $f$ belong to $D_k$, where $D_0 \subseteq D_k$ is an irreducible radical tower.
\end{definition}

\begin{lemma}[Abel]\label{le:Abel}
Let $p$ be a prime and $K$ be a field. The polynomial $x^p-c \in K[x]$ is irreducible over $K$ when $c$ is not a $p$th power of $K$.
\end{lemma}
\begin{proof}
See~\cite[p.~118, Abel's lemma]{dorrie1965100}.
\end{proof}

\begin{corollary}\label{co:Abel}
Let $p$ be a prime. Assume that $x^p-c\in K[x]$, where $K$ is a field. If ${\mathrm{e}}^{{2 \pi \mathrm{i}}/{p}} \in K$ and $u$ is a root of $x^p-c$, then $K \subseteq K(u)$ satisfies $(1,1)$ or $(u,p)$.
\end{corollary}
\begin{proof}
If $x^p-c \in K[x]$ is reducible over $K$, by Lemma~\ref{co:Abel}, then we can find $\beta \in K$ such that ${\beta}^p=c$. So we have $x^p-c = (x-\beta) (x-\beta {\mathrm{e}}^{{2 \pi \mathrm{i}}/{p}}) \cdots (x-\beta {\mathrm{e}}^{{2 \pi \mathrm{i} (p-1)}/{p}})$. Since $u$ is a root of $x^p-c$, it follows that $K \subseteq K(u)$ satisfies condition $(1,1)$. If $x^p-c \in K[x]$ is irreducible over $K$, by Definition~\ref{def:irreducible radical tower}, then we have $K \subseteq K(u)$ satisfies condition $(u,p)$.
\end{proof}

\begin{remark}\label{re:Abel}
Let $p \geq 2$ be an integer. Assume that $K$ is a field such that for any prime $q \leq p$, ${\mathrm{e}}^{{2 \pi \mathrm{i}}/{q}} \in K$ holds. If $a^p \in K$, then $K \subseteq K(a)$ is an irreducible radical tower.
\end{remark}
\begin{proof}
Let $p_0=1$. We write $p=\prod_{j=1}^{k} p_j$, where $p_j$ is a prime. Then we have
\[
K(a^{p_0 p_1 \cdots p_{k-1} p_k}) \subseteq K(a^{p_0 p_1 \cdots p_{k-1}}) \subseteq \cdots \subseteq K(a^{p_0 p_1}) \subseteq K(a^{p_0}).
\]
By Corollary~\ref{co:Abel}, for each $1 \leq j \leq k$, the extension $K(a^{p_0 p_1 \cdots p_{j-1} p_j}) \subseteq K(a^{p_0 p_1 \cdots p_{j-1}})$ satisfies $(1,1)$ or $(a^{p_0 p_1 \cdots p_{j-1}},p_j)$, so that $K \subseteq K(a)$ is an irreducible radical tower.
\end{proof}

\section{An irreducible radical tower with enough roots of unity}\label{sec:3}
This section aims to use Theorem~\ref{th:Gauss} to show the equivalence of Definition~\ref{def:solvable by radicals} to the definitions of algebraic solvability in~\cite{rosen1995niels,weintraub2009galois}. Corollary~\ref{co:Gauss} will achieve this goal.

\begin{theorem}[Gauss\footnote{We thank an anonymous reviewer for pointing out that a special case of Theorem~\ref{th:Gauss}, which is all cyclotomic fields are contained in radical towers, is a very classical result from Gauss~\cite[Chapter 7, \S 359]{gauss1996disquisitiones}.}]\label{th:Gauss}
Let $q \geq 3$ be a prime and $E$ be a field. We can find an irreducible radical tower $E \subseteq K_q^E$, and it satisfies
\[
E_0 = E \subseteq E_1 \subseteq E_2 \subseteq \cdots \subseteq E_k =K_q^E \supseteq E({\mathrm{e}}^{\frac{2 \pi \mathrm{i}}{3}}, {\mathrm{e}}^{\frac{2 \pi \mathrm{i}}{5}}, \ldots, {\mathrm{e}}^{\frac{2 \pi \mathrm{i}}{q}})
\]
such that ${\mathrm{e}}^{{2 \pi \mathrm{i}}/{q_{j}}} \in E_{j-1}$, and $E_{j-1} \subseteq E_{j}$ satisfies $(c_{j}, q_{j})$ for $1 \leq j \leq k$.
\end{theorem}
\begin{proof}
See Appendix~\ref{app:2}.
\end{proof}

\begin{corollary}\label{co:Gauss}
Let $F$ be a field and $f \in F [x]$. Assume that all the roots of $f$ belong to $F_n$, where
\[
F \subseteq F_1 \subseteq F_2 \subseteq \cdots \subseteq F_{c-1} \subseteq F_c \subseteq \cdots \subseteq F_n
\]
is such that for each $1 \leq j \leq n$, we have $u_{j}^{b_{j}} \in F_{j-1}$, $b_{j} \in \mathbb{N_+}$, and $F_{j} = F_{j-1} (u_{j})$. Then we can find an irreducible radical tower
\[
E_0 = F \subseteq E_1 \subseteq E_2 \subseteq \cdots \subseteq E_{k-1} \subseteq E_k \cdots \subseteq E_d \supseteq F_n
\]
such that ${\mathrm{e}}^{{2 \pi \mathrm{i}}/{q_{j}}} \in E_{j-1}$, and $E_{j-1} \subseteq E_{j}$ satisfies $(a_{j}, q_{j})$ for $1 \leq j \leq d$.
\end{corollary}
\begin{proof}
Let $q$ be the smallest prime larger than $\max_{1 \leq j \leq n}\{ b_j \}$. By Theorem~\ref{th:Gauss}, we can find an irreducible radical tower
\[
E_0 = F \subseteq E_1 \subseteq E_2 \subseteq \cdots \subseteq E_k \supseteq E({\mathrm{e}}^{\frac{2 \pi \mathrm{i}}{3}}, {\mathrm{e}}^{\frac{2 \pi \mathrm{i}}{5}}, \ldots, {\mathrm{e}}^{\frac{2 \pi \mathrm{i}}{q}}).
\]
We add $u_1, u_2, \ldots, u_n$ to $E_k$ one by one. Since ${\mathrm{e}}^{{2 \pi \mathrm{i}}/{p}} \in E_k$ holds for any prime $p \leq q$, by Remark~\ref{re:Abel} we can find an irreducible radical tower
\[
E_0 = F \subseteq E_1 \subseteq E_2 \subseteq \cdots \subseteq E_{k-1} \subseteq E_k \cdots \subseteq E_d \supseteq F_n
\]
such that ${\mathrm{e}}^{{2 \pi \mathrm{i}}/{q_{j}}} \in E_{j-1}$, and $E_{j-1} \subseteq E_{j}$ satisfies $(a_{j}, q_{j})$ for $1 \leq j \leq d$.
\end{proof}

\section{A complete decomposition theorem on polynomial of prime degree}\label{sec:4}
\begin{lemma}[\cite{dorrie1965100}, p.~123, Theorem~IV]\label{le:reducible}
Assume that $f, g \in E[x]$ are irreducible over field $E$, and $\deg(f)$ is a prime. Let $x_f$ be a root of $f$ and $x_g$ be a root of $g$. If $f$ is reducible over $E(x_g)$, then $\deg(f) \mid \deg(g)$.
\end{lemma}
\begin{proof}
Let $p=\deg(f)$, $q=\deg(g)$. Since $p \geq 2$ is prime and $f$ is not irreducible over $E(x_g)$, we have $p \nmid [E(x_f, x_g): E(x_g)]$. Hence, because $p$ is a prime and
\[
[E(x_g, x_f): E(x_f)] \cdot p=[E(x_f, x_g): E(x_g)] \cdot q,
\]
it follows that $p \mid q$.
\end{proof}

We use Lemma~\ref{le:reducible} frequently to determine the irreducibility of polynomials in this article. Theorem~\ref{th:complete decomposition} shows a case that the $f$ in Lemma~\ref{le:reducible} can be factored into linear factors.

\begin{theorem}\label{th:complete decomposition}
Assume that $f, g \in E[x]$ are irreducible over field $E$, $p=\deg(f)$ and $q=\deg(g)$ are both primes. Assume that all the roots of $g$ are $y_1, y_2, \ldots, y_q$ such that $E(y_1) = E(y_2) = \cdots = E(y_q)$. If $f$ is reducible over $E(y_1)$, then $p=q$, $f$ can be factored into linear factors over $E(y_1)$, and all the roots of $f$ can be expressed as
\[
x_j = \sum\limits_{t=0}^{p-1} w_t y_j^t,
\]
where $w_t \in E$, $j = 1, 2, \ldots, p$.
\end{theorem}
\begin{proof}
See Appendix~\ref{app:3}.
\end{proof}

\section{An irreducible complex conjugate closed radical tower}\label{sec:5}
\begin{definition}\label{def:complex conjugate closed field}
Let $F$ be a field. If for any $t \in F$, we also have $\overline{t} \in F$, then we call $F$ a \textit{complex conjugate closed field}.
\end{definition}

\begin{theorem}\label{th:complex conjugate closed radical tower}
Let $E$ be a complex conjugate closed field. Assume that $f \in E[x]$ is irreducible over $E$ and has degree $n \geq 2$. If $f$ is solvable by radicals over $E$, then we can find an irreducible radical tower $E \subseteq K$ such that
\begin{enumerate}
\item The field $K$ is a complex conjugate closed field with ${\mathrm{e}}^{{2 \pi \mathrm{i}}/{q}} \in K$, where $q$ is a prime;
\item The polynomials $f(x),x^q-{\alpha}^q \in F[x]$ are irreducible over $K$ but reducible over $K(\alpha)$;
\item When ${\alpha} \notin \mathbb{R}$, we have ${\alpha}\overline{\alpha} \in K$.
\end{enumerate}

\end{theorem}
\begin{proof}
By Corollary~\ref{co:Gauss}, if $f \in E[x]$ is solvable by radicals over $E$, we can find
\[
H_0=E \subseteq H_1 \subseteq H_2 \subseteq \cdots \subseteq H_k \cdots \subseteq H_d
\]
such that ${\mathrm{e}}^{{2 \pi \mathrm{i}}/{q_{j}}} \in H_{j-1}$, and $H_{j-1} \subseteq H_{j}$ satisfies $(c_{j}, q_{j})$ for $1 \leq j \leq d$. 

Then we consider the following field tower
\begin{equation}\label{eq:5.1}
F_0=E \subseteq F_1 \subseteq F_2 \subseteq \cdots \subseteq F_{2d}
\end{equation}
such that $F_{2j-1}=F_{2j-2} ( c_j \overline{c_j} )$ and $F_{2j}=F_{2j-1} (c_j)$ for $1 \leq j \leq d$. By Corollary~\ref{co:Abel}, we have \eqref{eq:5.1} is an irreducible radical tower. 

Since $F_0$ is a complex conjugate closed field, $F_{2j-1}=F_0 (c_1, \overline{c_1}, \ldots, c_{j-1}, \overline{c_{j-1}}, c_j \overline{c_j})$, $F_{2j}=F_0 (c_1, \overline{c_1}, \ldots, c_{j-1}, \overline{c_{j-1}}, c_{j}, \overline{c_{j}})$, we have $F_v$ are complex conjugate closed fields for $1 \leq v \leq 2d$. 

Since $F_{2d} \supseteq E_m$ contains all the roots of $f$, we can find $v_0$ ($1 \leq v_0 \leq 2d$) that $f$ is irreducible over $F_{v_0-1}$ but reducible over $F_{v_0}$. Then, we can set
\[
K=F_{{v_0}-1}, \quad \alpha=c_{\lfloor\frac{{v_0}+1}{2}\rfloor} (\overline{c_{\lfloor \frac{{v_0}+1}{2} \rfloor}})^{\frac{1-(-1)^{v_0}}{2}}, \quad q=q_{\lfloor\frac{{v_0}+1}{2}\rfloor}.
\]
If ${\alpha} \notin \mathbb{R}$, then
${\alpha}=c_{\lfloor {({v_0}+1})/{2} \rfloor}$. By the construction of~\eqref{eq:5.1}, we have ${\alpha}\overline{\alpha} \in K$.
\end{proof}

\section{Proof of Kronecker's solvability theorem}\label{sec:6}
The following Theorem~\ref{th:Kronecker}, which contains Kronecker's solvability theorem, is a modern version of the another proposition of Kronecker. Loewy's 1923 paper~\cite{loewy1923uber} gives the general odd degree case (see~\cite[p.~255, Loewy's theorem]{jensen2004number}).

\begin{theorem}[Kronecker]\label{th:Kronecker}
Let $E$ be a complex conjugate closed field. Assume that $f \in E[x]$ is irreducible over $E$ and has prime degree $p \geq 3$. If $f$ is solvable by radicals over $E$ and has a pair of complex conjugate roots, then $f$ only has one real root.
\end{theorem}
\begin{proof}
By Theorem~\ref{th:complex conjugate closed radical tower}, there is an irreducible radical tower $E \subseteq K$, where $K$ is a complex conjugate closed field, ${\mathrm{e}}^{{2 \pi \mathrm{i}}/{q}} \in K$, $f$ is irreducible over $K$ but reducible over $K(\alpha)$. In this instance, $q$ is prime, $\alpha$ is a root of $x^q-{\alpha}^q \in K[x]$, and $x^q-{\alpha}^q$ is irreducible over $K$. Additionally, if ${\alpha} \notin \mathbb{R}$, then
${\alpha}\overline{\alpha} \in K$. Then we have
\[
K(\alpha {\mathrm{e}}^{\frac{2 \pi \mathrm{i}}{p}}) = \cdots = K(\alpha {\mathrm{e}}^{\frac{2\pi \mathrm{i}(p-1)}{p}}) = K(\alpha).
\]
By Theorem~\ref{th:complete decomposition}, all the roots of $f$ can be expressed as
\begin{equation}\label{eq:Dorrie}
x_j = \sum\limits_{t=0}^{p-1} w_t {\left( \alpha {\mathrm{e}}^{\frac{2\pi \mathrm{i} j}{p}} \right)}^{t},
\end{equation}
where $w_t \in K$, $j=1, 2, \ldots, p$. For any $k \in \mathbb{Z}$, we have $x_k = x_{k'}$, where $1 \le k' \le p$, and $k \equiv k' \bmod p$. Since $f$ has a pair of complex conjugate roots, we denote them as $x_g, x_l$. Thus, we have $x_g = \overline{x_l}$.\footnote{Since $f \in E[x]$ is irreducible over $E$ and has a pair of complex conjugate roots; it follows that $f \in \mathbb{R}[x]$. Thus $f$ has a real root $x_{k_{0}}$. D{\"o}rrie studies $x_{k_{0}}=\overline{x_{k_{0}}}$, so his proof is shorter than ours. However, D{\"o}rrie and we all indirectly use the fact~\cite[pp.~430--432]{galois1846conditions} that $\operatorname{Gal}_{E} (f)$ can be generated by two permutations $x_{j} \mapsto x_{j+1}$ and $x_{j} \mapsto x_{c j}$ for some integer $c \not\equiv 0 \bmod p$, but very hidden.} We obtain
\begin{equation}\label{eq:6.1}
\sum\limits_{t=0}^{p-1} w_t {\left( \alpha {\mathrm{e}}^{ \frac{2 \pi \mathrm{i} g}{p} } \right)}^t = \sum\limits_{t=0}^{p-1} \overline{w_t} {\left( \overline{\alpha} {\mathrm{e}}^{-\frac{2\pi \mathrm{i}l}{p}} \right)}^t.
\end{equation}
We have two cases.

CASE I. When $\alpha \in \mathbb{R}$, then by~\eqref{eq:6.1} we have
\begin{equation}\label{eq:6.3}
\sum\limits_{t=0}^{p-1} \left( w_t {\mathrm{e}}^{\frac{2 \pi \mathrm{i} g t}{p}} \right) {\alpha}^t = \sum\limits_{t=0}^{p-1} \left( \overline{w_t} {\mathrm{e}}^{-\frac{2\pi \mathrm{i} l t}{p}} \right) {\alpha}^t.
\end{equation}
Since $x^p-{\alpha}^p \in K[x]$ is irreducible over $K$, we have that $\{ 1, {\alpha}, \ldots, {\alpha}^{p-1} \}$ is a basis for $E({\alpha})$ over $E$. Since ${\mathrm{e}}^{{2 \pi \mathrm{i}}/{p}}, w_t, \overline{w_t} \in K$, we can change the $\alpha$ in~\eqref{eq:6.3} to $\alpha {\mathrm{e}}^{{2 \pi \mathrm{i} j}/{p}}$ for each $1 \leq j \leq p$. It follows that
\begin{equation}\label{eq:6.4}
x_{g+j} = \overline{x_{l-j}}.
\end{equation}
When $j \equiv ({l-g})/{2} \bmod p$, then by~\eqref{eq:6.4} we have $x_{({g+l})/{2}} = \overline{x_{({g+l})/{2}}}$. Hence $f$ has one real root. If $x_{l-j_0}$ is a real root, by~\eqref{eq:6.4}, then we have $x_{g+j_0} = x_{l-j_0}$. Because $f$ does not have repeated roots, we have $x_k = x_d$ if and only if $k \equiv d \mod p$. Thus $g+j_0 \equiv l-j_0 \bmod p$. Then $j_0 \equiv ({l-g})/{2} \bmod p$. It follows that $f$ has exactly one real root.

CASE II. When $\alpha \notin \mathbb{R}$, let $\beta =\alpha \overline{\alpha}$; then $\beta \in K$. By~\eqref{eq:6.1}, we have
\begin{equation}\label{eq:6.5}
\sum\limits_{t=0}^{p-1} \left( w_t {\mathrm{e}}^{ \frac{2 \pi \mathrm{i} g t}{p} } \right) {\alpha}^t = \sum\limits_{t=0}^{p-1} \left( \overline{w_t} {\mathrm{e}}^{-\frac{2 \pi \mathrm{i} l t}{p}} {\beta}^t \right) \frac{1}{{\alpha}^t}.
\end{equation}
Since $x^p-{\alpha}^p \in K[x]$ is irreducible over $K$, we have that $\{ 1, {\alpha}, \ldots, {\alpha}^{p-1} \}$ is a basis for $E({\alpha})$ over $E$. Since ${\mathrm{e}}^{{2\pi \mathrm{i}}/{p}}, w_t, \overline{w_t}, \beta \in K$, we can change the $\alpha$ in~\eqref{eq:6.5} to $\alpha {\mathrm{e}}^{{2 \pi \mathrm{i}j}/{p}}$ for each $1 \leq j \leq p$; it follows that
\begin{equation}\label{eq:6.6}
x_{g+j}=\overline{x_{l+j}}.
\end{equation}
For any $1 \leq j \leq p$, we have, by~\eqref{eq:6.6}, that
\[
\overline{x_{l+j}}={x_{l+(g-l)+j}} = \overline{x_{l+2(g-l)+j}} = \cdots = x_{l+p(g-l)+j} = x_{l+j}.
\]
Thus $f$ only has real roots, but this contradicts our premise.
\end{proof}

\begin{theorem}[Abel]\label{th:Abel}
The general quintic polynomial is not solvable by radicals.
\end{theorem}
\begin{proof}
Let $a,b$ be integers and satisfy $4^4 a^5 > 5^5 b^4$. Assume that $a,b$ are divisible by $p$, and $b$ is indivisible by $p^2$. D{\"o}rrie~\cite[p.~127]{dorrie1965100} proves that $x^5-a x-b \in \mathbb{Q}[x]$ is irreducible over $\mathbb{Q}$ and has a pair of complex conjugate roots and three real roots. Since $\mathbb{Q}$ is a complex conjugate closed field, by Theorem~\ref{th:Kronecker} we obtain that $x^5-a x-b$ is not solvable by radicals.
\end{proof}

\section{A supplement to D{\"o}rrie's proof}\label{sec:7}
We now briefly introduce the gap in D{\"o}rrie's proof. Assume that $f \in \mathbb{Q}[x]$ is irreducible over $\mathbb{Q}$ and has degree that is an odd prime $n$. If $f$ is \textit{algebraically soluble} (see~\cite[p.~123]{dorrie1965100}), then we can find a series of intermediate fields
\[
D_0 = \mathbb{Q} \subset D_1 \subset D_2 \subset \cdots \subset D_{l-1} \subset D_l \subset \cdots \subset D_k
\]
such that for each $1 \leq j \leq k$, we have $u_{j}^{b_{j}} \in D_{j-1}$, $D_{j} = D_{j-1} \left( u_{j} \right)$, where $b_{j}$ is a prime, $u_{j}^{b_{j}}$ is not a $b_{j}$th power of $D_{j-1}$, and a root of $f$ is in $D_k$.

We follow D{\"o}rrie's method, namely add the $n$th root of unity to $\mathbb{Q}$, and then do as in his assumption~\cite[pp.~123--124]{dorrie1965100}:
\begin{quotation}
Also, with each substituted radical of our series, which still does not allow division of $f(x)$, we will also substitute at the same time the complex conjugate radical. Though this may be superfluous, it can certainly do no harm.
\end{quotation}
Now we get a new series of intermediate fields
\begin{equation}\label{eq:7.1}
\mathbb{Q} \subseteq P_0 = P'_0 = \mathbb{Q}({\mathrm{e}}^{\frac{2 \pi \mathrm{i}}{n}}) \subseteq P_1 \subseteq P'_1 \subseteq P_2 \subseteq P'_2 \subseteq \cdots \subseteq P_k \subseteq P'_k
\end{equation}
such that for each $1 \leq j \leq k$, ${P'_j}={{P}_j} \left( \overline{u_j} \right)$ and $P_j = P'_{j-1} \left( u_j \right)$. D{\"o}rrie seems to assume that ``if $f$ is irreducible over $P_{j_0}$, then $f$ is irreducible over $P'_{j_0}$.'' However, he gives no proof. D{\"o}rrie seems not to realize that $f$ may be irreducible over $P_{j_0}$ but reducible over $P'_{j_0}$. In this case, the first intermediate field that makes $f$ reducible in~\eqref{eq:7.1} (D{\"o}rrie denotes this field by $\textfrak{K}$) may not be a complex conjugate closed field. Namely, his assertion~\cite[p.~126]{dorrie1965100} ``$\overline{K_v}$ of ${K_v}$ are also $\textfrak{K}$-numbers'' is unproven. The following is an example where $\textfrak{K}$ is not a complex conjugate closed field.

\begin{example}\label{exa:1}\label{exa:gap}
Let $\theta = {\mathrm{e}}^{{2 \pi \mathrm{i}}/{11}}$ and $f(x)=\prod_{j=1}^{5} (x-{\theta}^j-{\theta}^{-j})$. Then $f \in \mathbb{Q}[x]$ is irreducible over $\mathbb{Q}$. By Theorem~\ref{th:Gauss} and Lemma~\ref{le:reducible}, we can find an irreducible radical tower
\[
\mathbb{Q} \subset \mathbb{Q}({\mathrm{e}}^{\frac{2 \pi \mathrm{i}}{5}}) \subset \mathbb{Q}({\mathrm{e}}^{\frac{2 \pi \mathrm{i}}{5}}, {\mathrm{e}}^{\frac{2 \pi \mathrm{i}}{11}} \sqrt[11]{2}) \subset E_1 \subset E_2 \cdots \subset E_k
\]
such that all the roots of $f$ are in $E_k$, and $f$ is irreducible over $\mathbb{Q}({\mathrm{e}}^{{2 \pi \mathrm{i}}/{5}}, {\mathrm{e}}^{{2 \pi \mathrm{i}}/{11}} \sqrt[11]{2})$. By D{\"o}rrie's assumption, we add ${\mathrm{e}}^{-{2 \pi \mathrm{i}}/{11}} \sqrt[11]{2}$ to $\mathbb{Q}({\mathrm{e}}^{{2 \pi \mathrm{i}}/{5}}, {\mathrm{e}}^{{2 \pi \mathrm{i}}/{11}} \sqrt[11]{2})$. It follows that
\[
{\left( \frac{{{\mathrm{e}}^{\frac{2 \pi \mathrm{i}}{11}}}\sqrt[11]{2}}{{{\mathrm{e}}^{-\frac{2 \pi \mathrm{i}}{11}}} \sqrt[11]{2}}\right)}^6 = {\mathrm{e}}^{\frac{24 \pi \mathrm{i}}{11}} = {\mathrm{e}}^{\frac{2 \pi \mathrm{i}}{11}} = \theta.
\]
Thus $f$ can be factored into linear factors over $\mathbb{Q}({\mathrm{e}}^{{2 \pi \mathrm{i}}/{5}}, {\mathrm{e}}^{{2\pi \mathrm{i}}/{11}} \sqrt[11]{2}, {\mathrm{e}}^{-{2\pi \mathrm{i}}/{11}} \sqrt[11]{2})$. Hence, we have $\textfrak{K} = \mathbb{Q}({\mathrm{e}}^{{2\pi \mathrm{i}}/{5}}, {\mathrm{e}}^{{2 \pi \mathrm{i}}/{11}} \sqrt[11]{2})$. Unfortunately, $\mathbb{Q}({\mathrm{e}}^{{2 \pi \mathrm{i}}/{5}}, {\mathrm{e}}^{{2 \pi \mathrm{i}}/{11}} \sqrt[11]{2})$ is not a complex conjugate closed field. So D{\"o}rrie's proof needs a supplement.
\end{example}

Now let us fill D{\"o}rrie's gap. Let $E=\mathbb{Q}$. The proof of Theorem~\ref{th:complex conjugate closed radical tower} shows a correct way to ``substitute the complex conjugate radical.'' Let $\textfrak{K}=K$ and $\lambda=\alpha$ (see~\cite[p.~124]{dorrie1965100}). Then $\textfrak{K}$ is a complex conjugate closed field, and D{\"o}rrie's proof now works correctly.

\section{Acknowledgments}
We are grateful to anonymous referees for their careful reading of the manuscript and helpful suggestions and comments. We wish to thank an anonymous editor and Qingshan Zhang for their valuable help in writing and communication.

\appendix

\section{Proof of~\eqref{eq:KroDor}}
The following proof of~\eqref{eq:KroDor} also closes D{\"o}rrie's gap.
\begin{proof}[\bf Proof of~\eqref{eq:KroDor}]\label{app:1}
Since $f$ is an irreducible cyclic polynomial with coefficients in $\mathbb{Q} \subseteq \mathbb{R}$, we have $\mathbb{Q}(r_1)=\mathbb{Q}{(\overline{r_1}})=\mathbb{Q}(r_1,r_2,\ldots,r_v)$. Then $x^{\mu} - w^{\mu} \in \mathbb{Q}(r_1)[x]$, and $\mathbb{Q}({\mathrm{e}}^{{2 \pi \mathrm{i}}/{\mu}}, r_1)$ is a complex conjugate closed field.

By Lemma~\ref{le:Abel}, we have $x^{\mu} - w^{\mu} \in \mathbb{Q}(r_1)[x]$ is irreducible over $\mathbb{Q}(r_1)$. Let $u(x)$ be the minimal polynomial of ${\mathrm{e}}^{{2 \pi \mathrm{i}}/{\mu}}$ over $\mathbb{Q}(r_1)$. Then $\mu \nmid \deg(u)$. Hence, by Lemma~\ref{le:reducible} we have $x^{\mu} - w^{\mu}$ is irreducible over $\mathbb{Q}({\mathrm{e}}^{{2 \pi \mathrm{i}}/{\mu}}, r_1)$.

Since $\deg(g) \nmid \deg(f)$, by Lemma~\ref{le:reducible} we have $g$ is irreducible over $\mathbb{Q}(r_1)$. Since $\deg(g) \nmid \deg(u)$, by Lemma~\ref{le:reducible} we have $g$ is irreducible over $\mathbb{Q}({\mathrm{e}}^{{2 \pi \mathrm{i}}/{\mu}}, r_1)$. By Theorem~\ref{th:KroEdw}, we have $g$ is reducible over $\mathbb{Q}({\mathrm{e}}^{{2 \pi \mathrm{i}}/{\mu}},r_1,w)$. 

According to the results above, by Theorem~\ref{th:KroEdw} and the proof of Theorem~\ref{th:complete decomposition}, then we have that all the roots of $g$ can be expressed by
\[
x_j = c_0 + c_1 {w {\mathrm{e}}^{\frac{2\pi \mathrm{i} j}{\mu}}} + c_2 {( w {\mathrm{e}}^{\frac{2\pi \mathrm{i} j}{\mu}})}^2 + \cdots + c_{\mu-1} {( w {\mathrm{e}}^{\frac{2\pi \mathrm{i} j}{\mu}})}^{\mu-1},
\]
where $c_0,c_1,\ldots,c_{\mu-1} \in \mathbb{Q}(r_1)$.
\end{proof}

\section{Proof of Theorem~\ref{th:Gauss}}\label{app:2}
\begin{proof}[\bf Proof of Theorem~\ref{th:Gauss}]
We use induction (on all primes $q \geq 3$) to prove this proposition.

Since ${\mathrm{e}}^{{2 \pi \mathrm{i}}/{3}}$ is a root of $x^2+x+1$, we have that $E \subseteq E({\mathrm{e}}^{{2\pi \mathrm{i}}/{3}})$ is an irreducible radical tower. Then we set $K_3^E = E({\mathrm{e}}^{{2 \pi \mathrm{i}}/{3}})$; it follows that Theorem~\ref{th:Gauss} is true for $q=3$.

We denote the largest prime less than $p$ as $m(p)$. Assume that Theorem~\ref{th:Gauss} is true for $q=m(p)$, where prime $p \geq 5$. We prove that Theorem~\ref{th:Gauss} is true for $q=p$. By Remark~\ref{re:Abel}, we get that $K_{m(p)}^E \subseteq K_{m(p)}^E ({\mathrm{e}}^{{2 \pi \mathrm{i}}/({p-1})})$ is an irreducible radical tower. Let $\tau$ be a primitive root modulo $p$, $\omega = {\mathrm{e}}^{{2 \pi \mathrm{i}}/{p}} $, ${\varepsilon}_j = {\mathrm{e}}^{{2 \pi \mathrm{i} j}/({p-1})}$, and ${\omega}^{[n]}= {\omega}^{{\tau}^n}$. We set the Lagrange resolvent
\[
\rho (\theta, {\varepsilon}_j) = {\theta}^{{\tau}^0} + {\varepsilon}_j {\theta}^{{\tau}^1} + {\varepsilon}_j^2 {\theta}^{{\tau}^2}+\cdots + {\varepsilon}_j^{p-2} {\theta}^{{\tau}^{p-2}}.
\]
Then we have
\[
\rho ({\omega}^{[n]}, {\varepsilon}_j) = {\omega}^{[n+0]} + {\varepsilon }_j {\omega}^{[n+1]} + {\varepsilon}_j^2 {\omega}^{[n+2]} + \cdots +{\varepsilon}_j^{p-2} {\omega}^{[n+p-2]}
\]
and
\[
\rho ({\omega}^{[n]}, {\varepsilon}_j) = \varepsilon_j^{-n}\rho ({\omega}^{[0]}, {\varepsilon}_j).
\]
For each $1 \leq j \leq p-1$, we have
\[
{\left[ {\rho} ({\omega}^{[0]}, {\varepsilon}_j) \right]}^{p-1} = {{\varepsilon}_j}^{\frac{(p-1)(p-2)}{2}} \prod\limits_{k=0}^{p-2} \rho ({\omega}^{[k]}, {\varepsilon}_j) \in K_{m(p)}^E ({\mathrm{e}}^{\frac{2 \pi \mathrm{i}}{p-1}}).
\]
Let $U^E_{m(p)} = K_{m(p)}^E ({\mathrm{e}}^{{2 \pi \mathrm{i}}/({p-1})})$. Then by Remark~\ref{re:Abel}, we can obtain
\[
\rho ({\omega }^{[0]}, {\varepsilon }_1), \quad \rho ({\omega }^{[0]}, {\varepsilon }_2), \quad \ldots, \quad \rho ({\omega }^{[0]}, {\varepsilon }_{p-1})
\]
by an irreducible radical tower
\[
\begin{aligned}
U^E_{m(p)} 
\subseteq U^E_{m(p)} (\rho ({\omega}^{[0]}, {\varepsilon}_1) )
&\subseteq U^E_{m(p)} (\rho ({\omega}^{[0]}, {\varepsilon}_1), \rho ({\omega}^{[0]}, {\varepsilon}_2))\\ 
&\subseteq \cdots \subseteq U^E_{m(p)} (\rho ({\omega}^{[0]}, {\varepsilon }_1), \rho ({\omega}^{[0]}, {\varepsilon}_2)), \ldots, \rho ({\omega}^{[0]}, {\varepsilon}_{p-1})).
\end{aligned}
\]
We set $K_p^E = U^E_{m(p)} ( {\rho ({\omega}^{[0]}, {\varepsilon }_{1})}, {\rho ({\omega}^{[0]}, {\varepsilon }_{2})}, \ldots, \rho ({\omega}^{[0]}, {\varepsilon}_{p-1}))$. Since we have
\[
\sum\limits_{j=1}^{p-1} {\varepsilon}_j = 0, \quad \sum\limits_{j=1}^{p-1} {\varepsilon}_j^2 = 0, \quad \ldots, \quad
\sum\limits_{j=1}^{p-1} {\varepsilon}_j^{p-2} = 0,
\]
it follows that
\[
\omega = {\omega}^{[0]} = \frac{1}{p-1} \sum\limits_{j=1}^{p-1} \rho ({\omega }^{[0]}, {\varepsilon}_j) \in K_p^E.
\]
Now we get an irreducible radical tower $K_{m(p)}^E \subseteq K_p^E$ such that $ {\mathrm{e}}^{{2\pi \mathrm{i}}/{p}} \in K_p^E$. Since $E \subseteq K_{m(p)}^E$ is an irreducible radical tower, it follows that $E \subseteq K_p^E$ is the irreducible radical tower that we want. Theorem~\ref{th:Gauss} is proved.
\end{proof}

\section{Proof of Theorem~\ref{th:Gauss}}\label{app:3}
\begin{proof}[\bf Proof of Theorem~\ref{th:complete decomposition}]
By Lemma~\ref{le:reducible}, we have $p=q$. We denote by $\varphi(x, y_1)$ the monic factor of $f$ over $E(y_1)$ such that $\deg(\varphi(x, y_1))$ is minimum. Let $\psi(x, y_1)=f(x)/{\varphi(x, y_1)}$. Since $g \in E[x]$ is irreducible over $E$ and $f \in E[x]$, we have $f(x) = \varphi(x, y_j) \psi(x,y_j)$, where $1 \leq j \leq p$, $\varphi(x, y_j), \psi(x, y_j) \in E(y_j)[x]$. Hence, since $E(y_j)=E(y_1)$, $\deg(\varphi(x, y_j))=\deg(\varphi(x, y_1))$, we have $\varphi(x, y_j) \in E(y_1)[x]$ are irreducible over $E(y_1)$. 

We next prove that $\varphi(x, y_j)$ are distinct for distinct $j$. Since $\varphi(x, y_1) \notin E[x]$, we can find a coefficient $b_{y_1}$ of $\varphi(x, y_1)$ such that $b_{y_1} \notin E$. We change the $y_1$ in $\varphi(x, y_1)$ to $y_j$, and also, this change makes $b_{y_1}$ be changed to $b_{y_j}$.\footnote{Write $b_{y_1}=\sum_{k=0}^{p-1} c_k y_1^k$ with $c_k \in E$. Then $b_{y_j}=\sum_{k=0}^{p-1} c_k y_j^k$.} Let $R(x)=\prod_{j=1}^{p}(x-b_{y_j})$. Then we have $R \in E[x]$. If $R$ is reducible over $E$, since $p$ is a prime, $[E(y_j):E(b_{y_j})][E(b_{y_j}):E]=p$, we have $[E(b_{y_j}):E]=1$; contradicts $b_{y_1} \notin E$. Thus $R$ is irreducible over $E$, namely, $b_{y_j}$ are distinct for distinct $j$. It follows that $\varphi(x, y_j)$ are distinct for distinct $j$.

Since $\varphi(x, y_j) \in E(y_1)[x]$ are irreducible over $E(y_1)$, $\varphi(x, y_j)$ are distinct for distinct $j$, we have $\left( \varphi(x, y_1), \varphi(x, y_2), \ldots, \varphi(x, y_p)\right)=1$. (Note that $\varphi(x, y_j)$ are monic.)

If $\deg(\varphi(x, y_1)) \neq 1$, then we can find two positive integers $r_1, r_2$ such that
\[
p = \deg(\varphi(x, y_1)) r_1 + r_2, \quad r_2 < \deg(\varphi(x, y_1)).
\]
Since $\varphi(x, y_j) \in E(y_1)[x]$ for $1 \leq j \leq p$, $\left( \varphi(x, y_1), \varphi(x, y_2), \ldots, \varphi(x, y_{p})\right)=1$, it follows that
\[
f(x) / \prod\limits_{j=1}^{r_1} \varphi(x, y_j) \in E(y_1)[x], \quad \deg(f(x) / \prod\limits_{j=1}^{r_1} \varphi(x, y_j)) = r_2 < \deg(\varphi(x, y_1));
\]
contradicts the definition of $\varphi(x, y_1)$. Thus $\deg(\varphi(x, y_1)) = 1$.

Now we let
\[
F(x) = \prod\limits_{j=1}^{p} \varphi(x, y_j).
\]
Since $f \in E[x]$ is irreducible over $E$, $\varphi(x, y_1)$ divides $f(x)$, $F \in E[x]$, and $\deg(F) = p$, we have $F=f$. Since $\{ 1, y_1, \ldots, {y_1}^{p-1} \}$ is a basis for $E(y_1)$ over $E$, it follows that
\[
\varphi(x, y_j) = x-\sum\limits_{t=0}^{p-1} w_t y_j^t,
\]
where $w_t \in E$, $j = 1, 2, \ldots, p$. For convenience of notation, for each $1 \leq j \leq p$, let $x_j = x-\varphi(x, y_j)$; then $x_1, x_2, \ldots, x_p$ are the all roots of $f$.
\end{proof}

\end{document}